\documentclass{amsart}

\newtheorem{theorem}{Theorem}[section]
\newtheorem{lemma}[theorem]{Lemma}

\theoremstyle{definition}
\newtheorem{definition}[theorem]{Definition}

\theoremstyle{remark}
\newtheorem{remark}[theorem]{Remark}

\numberwithin{equation}{section}

\def\H {{\mathcal H}}
\def\Z {{\mathbb Z}}
\def\Q {{\mathbb Q}}

\begin{document}

\title[Multiplicative Bases]{Multiplicative Bases for the Centres of the Group Algebra and Iwahori-Hecke Algebra of the Symmetric Group}

\author{Andrew Francis}
\address{School of Computing, Engineering and Mathematics, University of Western Sydney, NSW 2751, Australia}
\email[Andrew~Francis]{a.francis@uws.edu.au}

\thanks{The first author was supported by Australian Research Council Future Fellowship FT100100898.}

\author{Lenny Jones}
\address{Department of Mathematics, Shippensburg University, Pennsylvania, USA}
\email[Lenny~Jones]{lkjone@ship.edu}
\thanks{Both authors have benefited greatly from the mentoring of Leonard Scott --- in the case of the first author as a postdoc mentor, and in the case of the second author as a PhD advisor.  They would like to record their gratitude for his patient and wise tutelage.}

\dedicatory{Dedicated to Leonard Scott on the occasion of his retirement.}

\begin{abstract}
Let $\H_n$ be the Iwahori-Hecke algebra of the symmetric group $S_n$, and let $Z(\H_n)$ denote its centre.
Let $B=\left\{b_1,b_2,\ldots ,b_t\right\}$ be a basis for $Z(\H_n)$ over $R=\Z[q,q^{-1}]$.
 Then $B$ is called \emph{multiplicative} if, for every $i$ and $j$, there exists $k$ such that $b_ib_j= b_k$.
 In this article we prove that no multiplicative bases for $Z(\Z S_n)$ and $Z(\H_n)$ when $n\ge 3$. In addition, we prove that there exist exactly two multiplicative bases for $Z(\Z S_2)$ and none for $Z(\H_2)$.
\end{abstract}

\maketitle

\section{Introduction}

Quite a bit is known about the integral linear structure of the centres of Iwahori--Hecke algebras of finite Coxeter groups over the ring $R=\Z [q,q^{-1}]$.  For instance, they are known to have a basis consisting of elements that specialize to conjugacy class sums~\cite{GR97}.
Our focus here is on the Iwahori--Hecke algebra $\H_n$ of the symmetric group, where there is the suggestion of a multiplicative structure through a connection with symmetric polynomials of Jucys--Murphy elements~\cite{DJ87,Mat99,FG:DJconj2006}.  This connection has yielded an alternative integral basis involving symmetric polynomials~\cite{FJ:newintbasis}, but as yet no multiplicative relations.

The ability to determine exactly how the product of elements of a basis for $Z(\H_n)$, the centre of $\H_n$, decomposes as a linear combination of basis elements would be particularly useful in any homomorphic context. For instance, the Brauer homomorphism plays an important role in the representation theory of the symmetric group, and has been generalized to $\H_n$~\cite{Jon87}. The fact that the codomain of this generalized Brauer homomorphism can be realized as certain products of elements from $Z(\H_n)$ was instrumental in developing a Green correspondence for $\H_n$~\cite{Du.GreenCor92}. However, as yet, an explicit determination of the structure constants has not been achieved~\cite{FW-centres-filtered-2008}. Such a description would provide additional insight into the representation theory of $\H_n$. To this end, it is reasonable to search for bases for $Z(\H_n)$ that have nice multiplicative properties.

In this paper we address the question put by Jie Du (private communication) of whether $Z(\H_n)$ might have a basis over $R$ that is \emph{multiplicative} (closed under the multiplication of the algebra).
We answer this question in the negative for $Z(\H_n)$ when $n\ge 3$ by showing that no multiplicative bases exist for $Z(\Z S_n)$, the centre of the group algebra. We also prove that there exist exactly two multiplicative bases for $Z(\Z S_2)$ and none for $Z(\H_2)$. We point out that multiplicative bases do exist for $Z(\Q S_n)$ and for $Z(\H_n)$ over $\Q(q)$ for all $n$. For example, nested sums of central primitive orthogonal idempotents do the job.

\section{Definitions and Notation}
Let $S_n$ be the symmetric group on $\{1,\dots,n\}$ with generating set of simple reflections
\[S:=\{(i\ \ i+1)\mid 1\le i\le n-1\}.\]
 Let $\Z S_n$ be the symmetric group algebra over the integers. The centre of $\Z S_n$, which we denote $Z(\Z S_n)$, has dimension equal to the number of partitions of $n$, and a $\Z$-basis consisting of conjugacy class sums. Considered over the rational numbers, $Z(\Q S_n)$ also contains a set of primitive orthogonal idempotents $e_\lambda$ in correspondence with the irreducible characters $\chi_{\lambda}$ for $\lambda\vdash n$.  These idempotents form another basis for the centre of $\Q S_n$, and can be expressed in terms of the irreducible characters as follows (see e.g.~\cite{CurtisReiner87v1}):
\[
e_\lambda=\frac{\chi_\lambda(1)}{n!}\sum_{w\in S_n}\chi_\lambda(w)w.
\]
Since the characters are constant on conjugacy classes, this expression can be re-written in terms of class sums $\underline{C}$:
\begin{equation}\label{idempotent:class.elts1}
e_\lambda=\frac{\chi_\lambda(1)}{n!}\sum_{C}\chi_\lambda({w_C})\underline{C}
\end{equation} 
where $w_C$ is a representative of the conjugacy class $C$.
\begin{remark}\label{Rem: Idempotent Notation}
  When the correspondence of a partition to its particular irreducible character, or conjugacy class is not needed, we simply use the notation $\chi_i$, $e_i$ or $C_i$ for an irreducible character, idempotent or conjugacy class, respectively. We also write $w_j$ for an element of the class $C_j$.
\end{remark}

Using Remark \ref{Rem: Idempotent Notation},
we can rewrite (\ref{idempotent:class.elts1}) as
\begin{equation}\label{idempotent:class.elts2}
e_i=\frac{\chi_i(1)}{n!}\sum_{j=1}^t\chi_i(w_j)\underline{C_j}.
\end{equation}
\begin{definition}
Let $R=\Z[q,q^{-1}]$, where $q$ is an indeterminate.
The \emph{Iwahori--Hecke algebra} $\H_n$ of $S_n$ is the unital associative $R$-algebra with generators $\{T_i\mid 1\le i\le n-1\}$, where $T_i$ corresponds to the simple reflection $(i\ \ i+1)\in S$, subject to the relations
\begin{align*}
 T_{i}T_{j}&=T_{j}T_{i} &\text{if $|i-j|\ge 2$}\ \\
 T_{i}T_{i+1}T_{i}&=T_{i+1}T_{i}T_{i+1} &\text{for $1\le i\le n-2$}\ \\
 T_{i}^2&=q+(q-1)T_{i}   &\text{for $1\le i\le n-1$.}
\end{align*}
\end{definition}

The Iwahori--Hecke algebra $\H_n$ is a deformation of the symmetric group algebra $\Z S_n$ \cite{CurtisReiner87v2}. In particular, the specialization of $\H_n$ at $q=1$ is isomorphic to $\Z S_n$. It follows that any basis for $Z(\H_n)$ over $R$ specializes at $q=1$ to a basis for $Z(\Z S_n)$. We denote the specialization of an element $h\in \H_n$, the specialization of every element of a subset $W$ of $\H_n$, or the specialization of every entry in a matrix $M$, at $q=1$ as $h|_{q=1}$, $W|_{q=1}$ or $M|_{q=1}$, respectively. Because of the isomorphism $\H_n|_{q=1}\simeq \Z S_n$, we think of the specialization of any element in $\H_n$ as being an element of $\Z S_n$.

\begin{definition}\label{Def:mult}
  A basis $B=\left\{b_1,b_2,\ldots ,b_t\right\}$ for $Z(\H_n)$ over $R$ is called \emph{multiplicative} if, for every $i$ and $j$, there exists $k$ such that $b_ib_j= b_k$. Specialization at $q=1$ gives analogous definitions for $Z(\Z S_n)$.
\end{definition}

\section{Multiplicative Bases}
The following preliminaries are needed to establish our main results.

\begin{lemma}\label{Lem:mult form}
     Let $\{e_1,e_2,\ldots ,e_t\}$ be a complete set of central primitive idempotents for $Z(\Q S_n)$.
  Let $B$ be a multiplicative basis for $Z(\Z S_n)$ and let $b\in B$. Then
  \[b=\sum_{i=1}^t\varepsilon_ie_i,\] where $\varepsilon_i \in \{-1,0,1\}$.
\end{lemma}
\begin{proof} Consider the sequence $b, b^2, b^3, b^4,\ldots $.  Since $B$ is  multiplicative and finite, there exist distinct nonnegative integers $k$ and $m$, and $\hat{b}\in B$, such that
\[b^{k}=\hat{b}=b^{m}.\]
Since $\left\{e_1,e_2,\ldots ,e_t\right\}$ is a basis for $Z(\Q S_n)$,
we can write $b=\sum_{i=1}^tw_ie_i$, where $w_i\in \Q$. Then
\[\sum_{i=1}^{t}w_i^{k}e_i=b^k=\hat{b}=b^m=\sum_{i=1}^{t}w_i^{m}e_i,\]
so that $w_i^{k}=w_i^{m}$ for all $i$. Hence, for $w_i\ne 0$, we have that $w_i^{k-m}=1$, which implies that $w_i=\pm 1$ and the lemma is proven.
\end{proof}

The following theorem, which we state without proof, is due to W. Burnside \cite{Isaacs}.
 \begin{theorem}[Burnside]\label{Thm:Burnside}
   Let $G$ be a finite group and let $\chi$ be a nonlinear irreducible character of $G$. Then there exists $g\in G$ such that $\chi(g)=0$.
 \end{theorem}

We are now in a position to prove the following theorem.
\begin{theorem}
There do not exist any multiplicative bases for either $Z(\H_n)$ over $R$ or for $Z(\Z S_n)$ when $n\ge 3$.
\end{theorem}

\begin{proof}
Since any basis for $Z(\H_n)$ over $R$ specializes at $q=1$ to a basis for $Z(\Z S_n)$, it is enough to show that no multiplicative basis exists for $Z(\Z S_n)$. Let $\{e_1,e_2,\ldots ,e_t\}$ be a complete set of central primitive idempotents for $Z(\Q S_n)$.
 By way of contradiction, assume that $B$ is a multiplicative basis for $Z(\Z S_n)$, and let $b\in B$. Considering $b$ as an element of $Z(\Q S_n)$, we have by Lemma \ref{Lem:mult form} and (\ref{idempotent:class.elts2}) that
   \begin{align*}
    b&=\sum_{i=1}^t\varepsilon_i e_i\\
    &=\sum_{i=1}^t\varepsilon_i\left(\frac{\chi_i(1)}{n!}\sum_{j=1}^t\chi_i(w_j)\underline{C_j}\right)\\
    &=\sum_{j=1}^t\left(\sum_{i=1}^{t}\varepsilon_i\frac{\chi_i(1)\chi_i(w_j)}{n!}\right)\underline{C_j},
    \end{align*} where $\varepsilon_i\in \{-1,0,1\}$. Since $n\ge 3$, there exists at least one nonlinear irreducible character of $S_n$. Without loss of generality, let $\chi_t$ be nonlinear. Thus, by Theorem \ref{Thm:Burnside}, there exists at least one conjugacy class, say $C_s$, such that $\chi_t(w_s)=0$. Then
    the absolute value of the coefficient on $\underline{C_s}$ in $b$ is
     \begin{align*}
       \left| \sum_{i=1}^{t-1}\varepsilon_i\frac{\chi_i(1)\chi_i(w_s)}{n!} \right|
   &=  \frac{1}{n!}\left| \sum_{i=1}^{t-1}\varepsilon_i\chi_i(1)\chi_i(w_s)\right|\\
       & \le  \frac{1}{n!}\sum_{i=1}^{t-1}\left|\chi_i(1)\right|\left|\chi_i(w_s)\right|\\
       & \le  \frac{1}{n!}\sum_{i=1}^{t-1}\chi_i^2(1)\\
       &<1.
     \end{align*}
     Since $b$ is integral, it follows that $\sum_{i=1}^{t-1}\varepsilon_i\frac{\chi_i(1)\chi_i(w_s)}{n!}=0$. But then, the central element $\underline{C_s}$ cannot be written as a linear combination of the elements of $B$, which contradicts the fact that $B$ is a basis for $Z(\Z S_n)$. \end{proof}
The case $n=2$ requires a separate analysis which we give in the following theorem.
 \begin{theorem}
There exist exactly two multiplicative bases for $Z(\Z S_2)$ and no multiplicative bases for $Z(\H_2)$ over $R$.
\end{theorem}
\begin{proof}

   It is easy to see that the two bases $\{(1), (12)\}$ and $\{(1), -(12)\}$ for $Z(\Z S_2)$ are multiplicative. To see that there are no others, let $B=\{b_1,b_2\}$ be a multiplicative basis for $Z(\Z S_2)$, where $b_1=c_1(1)+c_2(12)$ and $b_2=d_1(1)+d_2(12)$, with $c_1,c_2,d_1,d_2\in \Z$. Since $B$ is multiplicative, we have that $b_1^2=b_i$ and $b_2^2=b_j$ for some $i$ and $j$.

   Suppose that $b_1^2=b_1$ and $b_2^2=b_1$. Since $b_1^2=b_1$, equating coefficients gives $c_1^2+c_2^2=c_1$ and $2c_2c_2=c_2$, from which we conclude that $c_1=1$ and $c_2=0$. Thus $b_1=(1)$.  Then, since $b_2^2=b_1$, we have that either $d_1=\pm 1$ and $d_2=0$ or $d_1=0$ and $d_2=\pm 1$. The first case here implies that $b_2=(1)$, which is impossible since $B$ is a basis and $b_1=(1)$. The second case implies that $b_2=\pm (12)$, and so we get the two bases $\{(1),(12)\}$ and $\{(1),-(12)\}$. The other cases are similar and they yield no new bases.

   To prove that no multiplicative basis for $Z(\H_2)$ over $R$ exists is a bit more tedious. Note that $Z(\H_2)=\H_2$ and that $\{1,T_1\}$ is a basis for $\H_2$ over $R$. Suppose that $B=\{b_1,b_2\}$ is a multiplicative basis for $Z(\H_2)$ over $R$. Since $B$ must specialize at $q=1$ to a multiplicative basis for $Z(\H_2)$, we will assume that $b_1|_{q=1}=(1)$ and $b_2|_{q=1}=(12)$. The case that $b_2|_{q=1}=-(12)$ is similar. Then we can write
     \[b_1=f+gT_1\quad \mbox{and} \quad b_2=r+sT_1,\] where $f,g,r,s\in R$ with $f|_{q=1}=(1)=s|_{q=1}$ and $g|_{q=1}=0=r|_{q=1}$. Since $B$ is multiplicative, we have that $b_1^2=b_1$ or $b_1^2=b_2$. But
     \[b_1^2=f^2+g^2q+(2fg+g^2(q-1))T_1,\] and we see that $b_1^2|_{q=1}=(1)$, which implies that $b_1^2=b_1$. Similarly, $b_2^2=b_1$. Equating coefficients gives:
     \begin{align}
       f&=f^2+g^2q=r^2+s^2q\label{Eq:Coeff1}\\
       g&=2fg+g^2(q-1)=2rs+s^2(q-1).\label{Eq:Coeff2}
     \end{align}
     If $g\equiv 0$, then $b_1=fT_1$ and $f^2=f$ since $b_1^2=b_1$. Thus, $f\equiv 1$ and
     \begin{equation}\label{Eq:sumofsquares1}
     r^2+s^2q=1.
     \end{equation}
     Since $s|_{q=1}=1$, we have that $s\not \equiv 0$. Hence, from (\ref{Eq:Coeff2}), we get
     $2r+s(q-1)\equiv 0$, so that $r=-s(q-1)/2$. Substituting into (\ref{Eq:sumofsquares1}) gives $s^2=4/(q+1)^2\not \in R$, which is a contradiction.
     Thus, $g\not \equiv 0$ and from (\ref{Eq:Coeff2}) we have that $f=(1-g(q-1))/2$. Then substituting into (\ref{Eq:Coeff1}) gives $g^2=1/(q+1)^2\not \in R$, which completes the proof.
     \end{proof}
     \section{Comments and Conclusions}
     One wonders if some minor adjustment of the idea of a multiplicative basis produces any integral bases for $Z(\H_n)$ or for $Z(\Z S_n)$. Motivated by this speculation, we define the concept of a \emph{quasi-multiplicative} basis.
     \begin{definition}\label{Def:q-mult}
  A basis $B=\left\{b_1,b_2,\ldots ,b_t\right\}$ for $Z(\H_n)$ over $R$ is called \emph{quasi-multiplicative} if, for every $i$ and $j$, there exist $k$ and some nonzero element $f_{ij}\in R$, such that $b_ib_j=f_{ij}\cdot b_k$. Specialization at $q=1$ gives an analogous definition for $Z(\Z S_n)$.
\end{definition}
We see that if $f_{ij}=1$ for all $i$ and $j$ in Definition \ref{Def:q-mult}, then $B$ is a multiplicative basis. So Definition \ref{Def:q-mult} is a generalization of Definition \ref{Def:mult}, and every multiplicative basis is a quasi-multiplicative basis. In a forthcoming paper, we investigate the existence of new quasi-multiplicative bases for $Z(\H_n)$ and $Z(\Z S_n)$.

\bibliographystyle{amsalpha}

\end{document}